\documentclass[12pt]{article}

\usepackage{amssymb}

\newcommand{\qed}{$\;\;\;\Box$}
\newenvironment{proof}{\par\smallbreak{\sl Proof.~}}
{\unskip\nobreak\hfill \qed \par\medbreak}

\newcounter{claim}
\renewcommand{\theclaim}{\arabic{claim}}
{\par\medskip\par}

{\qed\par\smallbreak}
\newcommand{\hide}[1]{}

\newcommand{\Con}{{\mbox{C}}}
\newcommand{\D}{{\cal D}}

\newcommand{\N}{\mathbb{N}}
\newcommand{\R}{\mathbb{R}}

\newcommand{\LL}{{\cal L}}
\newcommand{\beq}{\begin{equation}}
\newcommand{\ee}{\end{equation}}

\renewcommand{\d}{\partial}

\newtheorem{thm}{Theorem}
\newtheorem{lemma}[thm]{Lemma}

\newtheorem{rem}[thm]{Remark}

\newcommand{\al}{\alpha}
\newcommand{\be}{\beta}
\newcommand{\ga}{\gamma}

\newcommand{\vphi}{\varphi}

\newcommand{\im}{\mathop{\rm im}}
\newcommand{\ind}{\mathop{\rm ind}}

\newcommand{\codim}{\mathop{\rm codim}}
\newcommand{\const}{\mathop{\rm const}}
\title{
On the Fredholm Solvability for a Class of 
Multidimensional Hyperbolic Problems
} 

\newcounter{thesame}
\author{
I.~Kmit\thanks{supported by a Humboldt Research Fellowship} \\ 
{\small
Institute for Applied Problems of Mechanics and Mathematics,}\\
{\small Ukrainian Academy of Sciences}\\
{\small Naukova St.\ 3b,
79060 Lviv,
Ukraine}
\\
{\small   E-mail:
{\tt kmit@informatik.hu-berlin.de}}
}

\date{}

\begin{document}

\maketitle

\begin{abstract}
\noindent 
We prove the Fredholm alternative for a class of two-dimensional first-order hyperbolic systems
with periodic-Dirichlet boundary conditions. 
Our approach is based on a regularization via
a right parametrix. 
\end{abstract}



\section{Introduction}\label{sec:intr} 

The Fredholm property of linearizations plays the key role in local investigations of nonlinear differential 
equations via the Implicit Function Theorem and the Lyapunov-Schmidt reduction (see, e.g., \cite{ChowHale,Ki}).  
In contrast to ODEs and  parabolic PDEs almost nothing is known about the Fredholm property for 
hyperbolic PDEs.

The Fredholm solvability for ODEs and many classes of parabolic PDEs can be easily derived
from the basic fact that Fredholm operators are exactly compact 
perturbations of  bijective operators. The hyperbolic case is much more complicated. 
It is well known that the inverse of a first-order hyperbolic operator decreases the 
smoothness. By this reason the Fredholm analysis of hyperbolic problems requires establishing an optimal
regularity relation between  the spaces of solutions and right-hand sides of the differential 
equations and finding an appropriate regularization to compensate the loss-of-smoothness effect.

In~\cite{KmRe1,KmRe2}  we  presented a quite general approach to proving Fredholmness
for first-order one-dimensional hyperbolic PDEs. It is based on the construction of a right regularizer 
(right parametrix) and using a functional-analytic criterion for Fredholmness in Banach spaces. 
The techniques of~\cite{KmRe1,KmRe2}  cover the so-called traveling-wave models from laser dynamics.
In the present paper we extend this approach (applying completely different techniques) to
 a class of multidimensional hyperbolic PDEs admitting an integral representation. 
Though we currently do not know any real physical interpretation for the problems
in this class, our results are interesting from the theoretical point of view
because the multidimensional case is qualitatively different. We demonstrate
a noteworthy effect that a higher dimension requires more regularization for the 
inverse operator (see Remark~\ref{rem:reg}). Another interesting feature of the hyperbolic systems 
under consideration is that the ``lower order'' terms are those terms contributing
into the system transversely to characteristic directions (c.f. the definition of the 
operator $D$ in (\ref{eq:CD})).

Specifically, we investigate a linear first-order two-dimensional hyperbolic system of the kind
\beq\label{eq:10}
\begin{array}{rr}
\displaystyle
\sum\limits_{j=1}^na_{ij}\left(
\al_i\partial_tu_j  + \partial_xu_j + \be_i\partial_yu_j + \ga_i(x,y,t)u_j
\right) + \sum\limits_{j=1}^nb_{ij}(x,y,t)u_j  = f_i(x,y,t),\\
i\le n,\quad (x,y,t)\in(0,1)\times\R\times\R,
\end{array}
\ee
supplemented with the  periodic conditions in $y$ and $t$
\beq\label{eq:11}
\begin{array}{l}
u_i\left(x,y+Y,t+T\right) = u_i(x,y,t),\qquad
i\le n,\quad (x,y,t)\in[0,1]\times\R\times\R,
\end{array}
\ee
and the Dirichlet boundary conditions in $x$
\beq\label{eq:12}
\begin{array}{l}
u_i(0,y,t) = 0,\qquad
i\le k,\quad (y,t)\in\R^2,\\
u_j(1,y,t) = 0,\qquad
k+1\le j\le n,\quad (y,t)\in\R^2.
\end{array}
\ee
Here  the periods $Y>0$ and $T>0$  and the coefficients $a_{ij}$, 
 $\al_i\ne 0$, and $\be_i\ne 0$ are fixed real constants, the coefficients
$\ga_i$, $b_{ij}: [0,1]\times\R\times\R\to\R$ and the right hand sides $f_i: [0,1]\times\R\times\R\to\R$ 
are known functions. Without loss of generality, consider the case
$n\ge 3$ and $2\le k\le n-1$. Fix an arbitrary
 $l\in\N$ in the range $1\le l\le k-1$ and suppose that the matrix $A=\left(a_{ij}\right)_{i,j=1}^n$
has the following diagonal-block structure
 \begin{equation}\label{eq:blockA}
  A=\left(
  \begin{array}{cccccc}
    A_1&0 & 0\\
    0& A_2& 0\\
    0& 0 & A_3
  \end{array}
  \right),
 \end{equation}
where $A_1$, $A_2$, and $A_3$ are $l\times l$, $(k-l)\times(k-l)$, and $(n-k)\times(n-k)$-matrices,
respectively, while $0$ denotes the null matrices of respective sizes. Moreover, the matrix
$B=\left(b_{ij}\right)_{i,j=1}^n$ is assumed to be one of the following two kinds:
\begin{equation}\label{eq:blockB}
  B=\left(
  \begin{array}{cccccc}
    0&0 & B_1\\
    B_2& 0& 0\\
    0& B_3 & 0
  \end{array}
  \right)
 \end{equation}
or
\begin{equation}\label{eq:blockBm}
  B=\left(
  \begin{array}{cccccc}
    0&\tilde B_1 & 0\\
    0& 0& \tilde B_2\\
    \tilde B_3& 0 & 0
  \end{array}
  \right),
 \end{equation} 
where $B_1, B_2, B_3, \tilde B_1, \tilde B_2$, and $\tilde B_3$ are, respectively,
$l\times(n-k)$, $(k-l)\times l$, $(n-k)\times(k-l)$, $l\times(k-l)$, $(k-l)\times(n-k)$, and
$(n-k)\times l$-matrices. For definiteness, we suppose (\ref{eq:blockB})
(the case of (\ref{eq:blockBm}) is quite similar).

We investigate solvability of the problem (\ref{eq:10})--(\ref{eq:12}) 
and state our result as a Fredholm alternative. More precisely, we prove
 that the problem is solvable iff 
the right hand side is orthogonal to all solutions 
to the homogeneous adjoint system
\beq
%\label{eq:10}
\begin{array}{rr}
\displaystyle
\sum\limits_{j=1}^na_{ji}\left(
-\al_j\partial_tu_j  - \partial_xu_j - \be_j\partial_yu_j 
+ \ga_j(x,y,t)u_j\right)
 + \sum\limits_{j=1}^nb_{ji}(x,y,t)u_j  = 0,\nonumber\\
i\le n,\quad (x,y,t)\in(0,1)\times\R\times\R,\nonumber
\end{array}\nonumber
\ee
endowed with conditions (\ref{eq:11}) and (\ref{eq:12}).

We will work within the algebra $C_{Y,T}([0,1]\times\R^2)$ of continuous functions which are 
$Y$-periodic in $y$ and $T$-periodic in $t$.
Let 
\beq\label{eq:W}
W\equiv\left(C_{Y,T}\left([0,1]\times\R^2\right)\right)^n
\ee
denote the space of right-hand sides endowed with
the usual max-norm and let
\begin{equation}\label{eq:13}
\begin{array}{cc}
\displaystyle
V\equiv \Bigl\{u\in W \,:\,u_i(0,y,t)=0 \mbox{ for } i\le k, u_i(1,y,t)=0 \mbox{ for } k+1\le i\le n,\nonumber\\
\displaystyle
\left[\sum\limits_{j=1}^na_{ij}\left(\al_i\partial_tu_j+\partial_xu_j+\be_i\partial_yu_j\right)\right]_{j=1}^n\in W \mbox{ for } i\le n
\Bigr\}\nonumber
\end{array}
\end{equation}
denote the space of solutions. Here 
$u=(u_1,\dots,u_n)$ and $\d_tu_j$, $\d_xu_j$, and $\d_yu_j$ are generalized derivatives. The space $V$
is endowed with the norm
\begin{equation}\label{eq:14}
\begin{array}{cc}
\displaystyle
\|u\|_{V}\equiv\|u\|_{W} + \left\|\left[\sum\limits_{i=1}^na_{ij}\left(\al_i\partial_tu_j+\partial_xu_j+\be_i\partial_yu_j\right)\right]_{j=1}^n
\right\|_{W}.
\end{array}
\end{equation}
Note that the space $V$ depends on the coefficients of
system (\ref{eq:10}). Notice also the continuous embedding
$$
C_{Y,T}^1\left([0,1]\times\R^2\right)\hookrightarrow V\hookrightarrow 
C_{Y,T}\left([0,1]\times\R^2\right).
$$

To state our result, let us introduce linear
operators
$C\in\LL(V;W)$ and $D\in\LL(W)$ by
\begin{equation}\label{eq:CD}
\begin{array}{cc}
\displaystyle
Cu
\equiv
\left[\sum\limits_{j=1}^na_{ij}\left(
\al_i\partial_tu_j+\partial_xu_j+\be_i\partial_yu_j+\ga_i(x,y,t)u_j\right)\right]_{i=1}^n,
\\
\displaystyle
Du
\equiv
\left[\sum\limits_{j=1}^n
b_{ij}(x,y,t)u_j\right]_{i=1}^n.
\end{array}
\end{equation}
The problem (\ref{eq:10})--(\ref{eq:12}) can now be written as 
$$
%\begin{equation}
%\label{t5}
Cu+Du=f.
$$
%\end{equation} 
In what follows, we also use notation
$$
\tilde\al_i=
\cases{0
&if
$b_{ij}\equiv 0$ for all $j\le n$,
 \cr \al_i
&otherwise
, \cr}
$$
$$
\tilde\be_i=
\cases{0
&if
$b_{ij}\equiv 0$ for all $j\le n$,
 \cr \be_i
&otherwise
. \cr}
$$

\begin{thm}\label{thm:fredh}
Suppose that problem (\ref{eq:10})--(\ref{eq:12}) satisfies the following assumptions:
\begin{equation}\label{eq:reg}
\ga_i\in L^\infty\left((0,1),C_{Y,T}^1\left(\R^2\right)\right),\quad
b_{ij}\in C^1_{Y,T}\left([0,1]\times\R^2\right),
\end{equation}
\begin{equation}\label{eq:det}
\det(a_{ij})_{i,j=1}^n\ne 0,
\end{equation}
 and
\begin{equation}\label{eq:ab}
(\tilde\be_i-\tilde\be_j)(\tilde\al_j-\tilde\al_s)-
(\tilde\be_j-\tilde\be_s)(\tilde\al_i-\tilde\al_j)\ne 0
\end{equation}
for all $i,j,s\in\{1,\dots,n\}$ with $1\le i\le l$, $l+1\le j\le k$, $k+1\le s\le n$
unless $\tilde\al_i=\tilde\al_j=0$.
Let $W$ and $V$ be function spaces defined by (\ref{eq:W}), (\ref{eq:13}), and
(\ref{eq:14}). Let $C\in\LL(V;W)$ and $D\in\LL(W)$ be linear operators defined by 
 (\ref{eq:CD}).
Then the following is true:

(i) The operator $C$  is an isomorphism from $V$ onto $W$.

(ii) The operator $C+D$ is a Fredholm operator from $V$ into $W$ with index zero.
\end{thm}
Part  $(i)$ of the theorem is a necessary starting point
of the Fredholm analysis. It shows  that the couple of spaces $(V,W)$ provides the desired optimal regularity 
relation between the solutions and the right-hand sides of the equations. 

\begin{rem}
Since the set of Fredholm operators is open,  the  conclusion of Theorem \ref{thm:fredh}  survives under small 
 perturbations in $L^\infty\left((0,1),C_{Y,T}^1\left(\R^2\right)\right)$ and in $C_{Y,T}^1\left([0,1]\times\R^2\right)$ of
the coefficients $\ga_i$ and $b_{jk}$, respectively. Such perturbations can modify the structure of matrix (\ref{eq:blockB}).
Thus,   the structure of (\ref{eq:blockB}) is not a necessary condition for the conclusion of the theorem
(though it is essential for our proof).
\end{rem}

In Section~\ref{sec:criter} we  prove a criterion of Fredholmness for linear operators in Banach spaces,
which is useful, in particular, for hyperbolic PDEs. Section~\ref{sec:spaces} is devoted to the desired properties of the
solution spaces. Our main result, Theorem~\ref{thm:fredh}, is proved in Section~\ref{sec:fredh}.

\section{Fredholmness criterion}\label{sec:criter}

Here we prove the following constructive Fredholmness criterion:

\begin{thm}\label{thm:criter}
Let $W$ be a Banach space, $I$ be the identity in $W$, and $K\in\LL(W)$ with $K^n$ being
compact for some $n\in\N$.  Then $I-K$ is a Fredholm operator of index zero.
\end{thm}

\begin{proof}
Since 
$$
I-K^n=(I-K)\sum\limits_{i=0}^{n-1}K^i,
$$
the sum $\sum_{i=0}^{n-1}K^i$ is a parametrix for the operator  $I-K\in\LL(W)$. Therefore,
the Fredholmness  of $I-K$ follows, i.e.  from \cite[Proposition 5.7.1]{Zeidler}
or \cite[Theorem 5.5]{Schechter}. Nevertheless, for the reader's convenience here
we give an independent, simple, and self-contained proof (of this fact). Our proof extends the argument that was used 
 in~\cite{KmRe1} in the case $n=2$. Note first that 
\beq
\label{eq:endl}
\dim\ker(I-K)\le \dim\ker(I-K^n)<\infty.
\ee
Similarly $\dim\ker(I-K)^*<\infty$, hence
$
\codim\overline{\im(I-K)}<\infty.
$
It remains to show that $\im(I-K)$ is closed.

Take a sequence $(w_j)\subset W$ and an element $w\in W$ such that
\begin{equation}\label{eq:4.1}
(I-K)w_j\to w.
\end{equation}
We have to show that $w \in \im(I-K)$.

By (\ref{eq:endl}) there exists a closed subspace $V$ of $W$ such that 
\beq
\label{eq:dirSum} 
W=\ker(I-K)\oplus V,
\ee
Consider the decomposition
$$
w_{j}=u_j+v_j, \mbox{ where } u_j\in\ker(I-K) \mbox{ and } v_j \in V.
$$
From (\ref{eq:4.1}) we infer that
\begin{equation}\label{eq:v}
(I-K)v_j\to w.
\end{equation}

Let us show that  the sequence $(v_j)$ is bounded. Suppose this is not true.
Without loss of generality we can assume that
\begin{equation}\label{eq:4.11}
\lim\limits_{j\to \infty}\|v_j\|=\infty.
\end{equation}
From (\ref{eq:v}) and (\ref{eq:4.11}) we get 
\beq
\label{eq:1Null}
(I-K)\frac{v_j}{\|v_j\|} \to 0,
\ee
hence
\beq
\label{eq:Null}
(I-K^n)\frac{v_j}{\|v_j\|}  \to 0.
\ee
Since $K^n$ is compact, there exist $v \in W$ and a subsequence 
$(v_{j_k})_{k \in \N}$ such that
\beq
\label{eq:conv}
K^n\frac{v_{j_k}}{\|v_{j_k}\|}  \to v.
\ee
The convergences (\ref{eq:conv}) and (\ref{eq:Null}) immediately imply that 
\beq
\label{eq:minus}
\frac{v_{j_k}}{\|v_{j_k}\|} \to v \in V.
\ee
Combining (\ref{eq:minus}) with (\ref{eq:1Null}), we get $(I-K)v=0$,
i.e. $v \in V \cap \ker (I-K)$ and $\|v\|=1$. This contradicts (\ref{eq:dirSum})
and proves  the boundedness of $(v_j)$.

Now we show that $w\in\im(I-K)$.
As
$K^n$ is compact, there exists $v\in W$ and a subsequence $(v_{j_k})$ such that
$
K^nv_{j_k}\to v
$
as $k\to\infty$.
By (\ref{eq:v}) we also have 
$$
(I-K^n)v_j=\sum\limits_{i=0}^{n-1}K^i(I-K)v_j\to  
\sum\limits_{i=0}^{n-1}K^iw.
$$ 
Therefore,
%,~(\ref{v}) yields
$$
\lim\limits_{k\to \infty}v_{j_k}=\sum\limits_{i=0}^{n-1}K^iw+v
$$
and
$$
w=\lim\limits_{k\to \infty}(I-K)v_{j_k}=(I-K)\left(\sum\limits_{i=0}^{n-1}K^iw+v\right)\in\im(I-K)
$$
as desired. The Fredholm property is thereby proved.

To prove that $I-K$ has index zero, we additionally use a homotopy argument.
Let us consider the continuous function
$$
s\in\R\mapsto I-sK\in\LL(W).
$$
Since $K^n\in\LL(W)$ is a compact operator,
the operators $\left(sK\right)^n\in\LL(W)$ are compact for each $s\in\R$ and, as we just proved,
the operators $I-sK$ are Fredholm. 
By \cite[Proposition 5.8.1]{Zeidler}, $\ind(I-sK)=\const$ for all $s\in\R$.
It remains to note that the identity operator $I$ has index zero.
\end{proof}

\section{More about solution spaces}\label{sec:spaces}

We now prove that $V$ is a Banach space.

\begin{lemma}\label{lem:2.2}
The space $V$ is complete.
\end{lemma}

\begin{proof}
Let $(u^{m})_{m\in\N}$ be a fundamental sequence
in $V$.
Then 
$$
(u^{m})_{m\in\N} \quad\mbox{ and }\quad
\left(\left[\sum\limits_{j=1}^na_{ij}(\al_i\d_tu_j^{m}+\d_xu_j^{m}+\be_i\d_yu_j^{m})\right]_{i=1}^n\right)_{m\in\N}
$$
are fundamental sequences in $W$. Due to the completeness of $W$,
there exist
$u,v\in W$  such that
$$
u^{m}\to u \; \mbox{ and } \;
\left[\sum\limits_{j=1}^na_{ij}(\al_i\d_tu^{m}+\d_xu^{m}+\be_i\d_yu^{m})\right]_{i=1}^n\to  v
\mbox{ in } W \mbox{ as } m\to\infty.
$$
It remains to show that
$
\left[\sum_{j=1}^na_{ij}(\al_i\d_tu_j+\d_xu_j+\be_i\d_yu_j)\right]_{i=1}^n=v
$
in the sense of generalized derivatives. Let $\langle\cdot,\cdot\rangle : \D^*\times\D\to\R$  denote
the dual pairing. Then for all
$\vphi_1,\dots,\vphi_n \in\D\left((0,1)\times\left(0,Y\right)\times(0,T)\right)$ we have
\begin{eqnarray*}
&\displaystyle
\sum\limits_{i=1}^n\left\langle\sum\limits_{j=1}^na_{ij}(\al_i\d_t+\d_x+\be_i\d_y)u_j,\vphi_i\right\rangle&\\
&\displaystyle
=
-\sum\limits_{i=1}^n\sum\limits_{j=1}^na_{ij}\left\langle u_j,(\al_i\d_t+\d_x+\be_i\d_y)\vphi_i\right\rangle
&\\
&\displaystyle
=
-\sum\limits_{i=1}^n
\sum\limits_{j=1}^na_{ij}\lim\limits_{m\to\infty}\left\langle u_j^{m},(\al_i\d_t+\d_x+\be_i\d_y)\vphi_i\right\rangle&\\
&\displaystyle
=\sum\limits_{i=1}^n\lim\limits_{m\to\infty}
\left\langle \sum\limits_{j=1}^n a_{ij}(\al_i\d_t+\d_x+\be_i\d_y)u_j^{m},\vphi_i\right\rangle=\sum\limits_{i=1}^n
\left\langle v_i,\vphi_i\right\rangle
&
\end{eqnarray*}
as desired. 
\end{proof}

\section{Fredholm alternative 
(proof of Theorem~\ref{thm:fredh})}\label{sec:fredh}

To prove part $(i)$ of the theorem, it is sufficient to show that, given $f\in W$, there exists a unique $u\in V$ satisfying the system
\beq\label{eq:15}
\sum\limits_{j=1}^na_{ij}\left(
\al_i\partial_tu_j  + \partial_xu_j + \be_i\partial_yu_j + \ga_i(x,y,t)u_j
\right) = f_i(x,y,t),\quad i\le n,
\ee
and the apriori estimate
\beq\label{eq:apr}
\|u\|_{V} \le C \|f\|_{W}
\ee
with a constant $C$ independent of $f$ and $u$. 
Rewrite (\ref{eq:15}) as
$$
\left(\frac{d}{d\xi}+\ga_i\right)\left[\sum\limits_{j=1}^na_{ij}u_j\left(\xi,
y+\be_i(\xi-x),t+\al_i(\xi-x)\right)\right]\bigg|_{\xi=x} = f_i(x,y,t), \quad i\le n.
$$
Set
$$
  A_0=\left(
  \begin{array}{cccccc}
    A_1&0 \\
    0& A_2
  \end{array}
  \right).
$$
Taking into account the structure of matrix $A$ (assumptions (\ref{eq:blockA})) and the non-degenerateness
of $A$ (assumption (\ref{eq:det})),  system (\ref{eq:15}) 
has a unique solution in $V$ explicitely given by the formula
\begin{eqnarray}
\label{eq:16}
\lefteqn{
u_i(x,y,t) = \frac{1}{\det A_0}\sum\limits_{j=1}^k\left(A_{ji}^0\right)^{ad}}\nonumber\\
\displaystyle
&&
\times
\int\limits_0^xE_j(\xi;x,y,t)
f_j(\xi,y+\be_j(\xi-x),t+\al_j(\xi-x))\,d\xi,\quad i\le k;\nonumber\\
\lefteqn{
u_i(x,y,t) = \frac{1}{\det A_3}\sum\limits_{j=k+1}^n\left(A_{ji}^3\right)^{ad}}\\
\displaystyle
&&
\times
\int\limits_x^1E_j(\xi;x,y,t)
f_j(\xi,y+\be_j(\xi-x),t+\al_j(\xi-x))\,d\xi,\quad k+1\le i\le n,\nonumber
\end{eqnarray}
where $\left\{(A_{ij}^s)^{ad}\right\}_{i,j}$ stands for the adjoint matrix to $A_s$ and
\beq\label{eq:EG}
\begin{array}{cc}
\displaystyle E_i(\xi;x,y,t)\equiv\exp\left\{\int_x^\xi
\Gamma_i(\xi_0;x,y,t)\,d\xi_0\right\},\\
\Gamma_i(\xi_0;x,y,t)\equiv \ga_i(\xi_0,y+\be_i(\xi_0-x),t+\al_i(\xi_0-x)).
\end{array}
\ee

It remains to prove (\ref{eq:apr}). As a straightforward consequence of
(\ref{eq:16}), we have
\begin{equation}\label{eq:apr_0}
\|u\|_{W} \le C \|f\|_{W},
\end{equation}
where the constant $C$ does not depend on $f$ and $u$. 
To derive  (\ref{eq:apr}) from (\ref{eq:apr_0}), it suffices to show that because
 $u$ defined by (\ref{eq:16}) satisfies (\ref{eq:10})
in a distributional sense, then
$\sum_{j=1}^na_{ij}\left(
\al_i\partial_tu_j  + \partial_xu_j + \be_i\partial_yu_j\right) =- \sum_{j=1}^na_{ij}\ga_i(x,y,t)u_j
  + f_i(x,y,t)$  is a known continuous function for each  $i\le n$.
To this end, notice that the function $\Gamma_i(\xi;x,y,t)$ satisfies the equation 
$\left(\al_i\partial_t  + \partial_x + \be_i\partial_y\right)\Gamma_i=0$ for a.a. $\xi\in(0,1)$.
Now, fix $i\le k$ (the case $k+1\le i\le n$ is similar). Take 
$\vphi\in\D\left((0,1)\times\left(0,Y\right)\times(0,T)\right)$ and choose
a $\Con^1\left([0,1]\times[0,Y]\times[0,T]\right)$-sequence $f_i^m\to f_i$ in $\Con\left([0,1]\times[0,Y]\times[0,T]\right)$ 
as $m\to\infty$. We have
\begin{eqnarray*}
\lefteqn{
\left\langle
\sum\limits_{j=1}^na_{ij}\left(
\al_i\partial_t  + \partial_x + \be_i\partial_y\right)u_j,\vphi
\right\rangle
=-
\left\langle
\sum\limits_{j=1}^na_{ij}u_j,\left(\al_i\partial_t  + \partial_x + \be_i\partial_y\right)\vphi
\right\rangle}
\\
\displaystyle
&&
=-
\left\langle
\int_0^xE_i(\xi;x,y,t)f_i(\xi,y+\be_i(\xi-x),t+\al_i(\xi-x))\,d\xi,
\left(\al_i\partial_t  + \partial_x + \be_i\partial_y\right)\vphi
\right\rangle
\\
\displaystyle
&&
=-
\lim\limits_{m\to\infty}\left\langle
\int_0^xE_i(\xi;x,y,t)f_i^m(\xi,y+\be_i(\xi-x),t+\al_i(\xi-x))\,d\xi,
\left(\al_i\partial_t  + \partial_x + \be_i\partial_y\right)\vphi
\right\rangle
\\
\displaystyle
&&
=\lim\limits_{m\to\infty}\langle
\int_0^x\int_x^\xi\left(\al_i\partial_t  + \partial_x + \be_i\partial_y\right)\Gamma_i(\xi_0;x,y,t)\,d\xi_0
\\
\displaystyle
&&\times
E_i(\xi;x,y,t)f_i^m(\xi,y+\be_i(\xi-x),t+\al_i(\xi-x))\,d\xi,\vphi
\rangle
\\
\displaystyle
&&
+\lim\limits_{m\to\infty}\left\langle
\int_0^xE_i(\xi;x,y,t)\left(\al_i\partial_t  + \partial_x + \be_i\partial_y\right)
f_i^m(\xi,y+\be_i(\xi-x),t+\al_i(\xi-x)),\vphi
\right\rangle 
\\
\displaystyle
&&
-\left\langle
\ga_i(x,y,t)\int_0^xE_i(\xi;x,y,t)f_i(\xi,y+\be_i(\xi-x),t+\al_i(\xi-x)),\vphi
\right\rangle + \left\langle f_i,\vphi\right\rangle 
\\
\displaystyle
&&
=-\left\langle\sum\limits_{j=1}^na_{ij}\ga_i(x,y,t)u_j,\vphi\right\rangle + \left\langle f_i,\vphi\right\rangle,
\end{eqnarray*}
where  the last equality holds by (\ref{eq:16}).
 The isomorphism property $(i)$ is thereby proved.

To prove part $(ii)$ of the theorem, note that $C+D\in\LL(V,W)$ is Fredholm  iff $I+DC^{-1}\in\LL(W)$ is Fredholm,
 where $I$ is the identity in $W$. The proof will be finished by setting $K=-DC^{-1}$ and applying Theorem~\ref{thm:criter}
with $n=3$. We only need to show that $K^3$ is compact. 

Take a bounded set $N\subset W$ and let $M$ be its image under $K^3$. To show that $M$ is precompact in 
$W$, we use Arzela-Ascoli precompactness criterion in $\Con\left([0,1]\times[0,Y]\times[0,T]\right)$.
As $K^3$ is a bounded operator on $W$, the set $M$ is uniformly bounded in $W$. It remains to
check the equicontinuity property of $M$ in $W$. Given $\overline u\in W$, set 
$\tilde u\equiv DC^{-1}D\overline u$. Using the representation (\ref{eq:15}) for $C^{-1}$, we get
the following equalities: if $i\le l$, then
\begin{eqnarray}
\lefteqn{
\tilde u_i(x,y,t) = 
\sum\limits_{j=k+1}^nb_{ij}(x,y,t)\frac{1}{\det A_3}
\sum\limits_{r=k+1}^n\left(A_{rj}^3\right)^{ad}}\nonumber\\
\displaystyle
&&\times
\int\limits_{x}^1E_r(\xi;x,y,t)
\sum\limits_{q=l+1}^k(b_{rq}\overline u_q)(\xi,y+\be_r(\xi-x),t+\al_r(\xi-x))\,d\xi;
\nonumber
\end{eqnarray}
if $l+1\le i\le k$, then
\begin{eqnarray}
\displaystyle
\lefteqn{
\tilde u_i(x,y,t) = 
\sum\limits_{j=1}^lb_{ij}(x,y,t)\frac{1}{\det A_1}
\sum\limits_{r=1}^l\left(A_{rj}^1\right)^{ad}}\nonumber\\
\displaystyle
&&\times
\int\limits_0^xE_r(\xi;x,y,t)
\sum\limits_{q=k+1}^n(b_{rq}\overline u_q)(\xi,y+\be_r(\xi-x),t+\al_r(\xi-x))\,d\xi;\nonumber
\end{eqnarray}
if $k+1\le i\le n$, then
\begin{eqnarray}
\displaystyle
\lefteqn{
\tilde u_i(x,y,t) = 
\sum\limits_{j=l+1}^kb_{ij}(x,y,t)\frac{1}{\det A_2}
\sum\limits_{r=l+1}^k\left(A_{rj}^2\right)^{ad}}\nonumber\\
\displaystyle
&&\times
\int\limits_0^xE_r(\xi;x,y,t)
\sum\limits_{q=1}^l(b_{rq}\overline u_q)(\xi,y+\be_r(\xi-x),t+\al_r(\xi-x))\,d\xi.
\nonumber
\end{eqnarray}
Now, given $f\in W$, let $\overline u=C^{-1}DC^{-1}f$. Note that for 
 $\tilde u$ defined by the formulas above
 we have $\tilde u=\left(DC^{-1}\right)^3f$. Furthermore, $\overline u$ is explicitely given by: if $i\le l$, then
\begin{eqnarray*}
\lefteqn{
\overline u_i(x,y,t) = 
\frac{1}{\det A_1}
\sum\limits_{j=1}^l\left(A_{ji}^1\right)^{ad}\int\limits_{0}^xE_j(\xi;x,y,t)}
\\
\displaystyle
&&\times
\sum\limits_{r=k+1}^nb_{jr}(\xi,y+\be_j(\xi-x),t+\al_j(\xi-x))\,d\xi
\frac{1}{\det A_3}
\sum\limits_{q=k+1}^n\left(A_{qr}^3\right)^{ad}\\
\displaystyle
&&\times
\int\limits_{\xi}^1E_q(\xi_1;\xi,y+\be_j(\xi-x),t+\al_j(\xi-x))
\\
\displaystyle
&&\times
f_q(\xi_1,y+\be_j(\xi-x)+\be_q(\xi_1-\xi),t+\al_j(\xi-x)+\al_q(\xi_1-\xi))\,d\xi_1;
\end{eqnarray*}
 if $l+1\le i\le k$, then
\begin{eqnarray*}
\lefteqn{
\overline u_i(x,y,t) = 
\frac{1}{\det A_2}
\sum\limits_{j=l+1}^k\left(A_{ji}^2\right)^{ad}
\int\limits_{0}^xE_j(\xi;x,y,t)}\\
\displaystyle
&&\times
\sum\limits_{r=1}^lb_{jr}(\xi,y+\be_j(\xi-x),t+\al_j(\xi-x))\,d\xi
\frac{1}{\det A_1}
\sum\limits_{q=1}^l\left(A_{qr}^1\right)^{ad}\\
\displaystyle
&&\times
\int\limits_{0}^\xi E_q(\xi_1;\xi,y+\be_j(\xi-x),t+\al_j(\xi-x))
\\
\displaystyle
&&\times
f_q(\xi_1,y+\be_j(\xi-x)+\be_q(\xi_1-\xi),t+\al_j(\xi-x)+\al_q(\xi_1-\xi))\,d\xi_1;
\end{eqnarray*}
 if $k+1\le i\le n$, then
\begin{eqnarray*}
\lefteqn{
\overline u_i(x,y,t) = 
\frac{1}{\det A_3}
\sum\limits_{j=k+1}^n\left(A_{ji}^3\right)^{ad}
\int\limits_{x}^1E_j(\xi;x,y,t)}
\\
\displaystyle
&&\times
\sum\limits_{r=l+1}^kb_{jr}(\xi,y+\be_j(\xi-x),t+\al_j(\xi-x))\,d\xi
\frac{1}{\det A_2}
\sum\limits_{q=l+1}^k\left(A_{qr}^2\right)^{ad}\\
\displaystyle
&&\times
\int\limits_{0}^\xi E_q(\xi_1;\xi,y+\be_j(\xi-x),t+\al_j(\xi-x))
\\
\displaystyle
&&\times
f_q(\xi_1,y+\be_j(\xi-x)+\be_q(\xi_1-\xi),t+\al_j(\xi-x)+\al_q(\xi_1-\xi))\,d\xi_1.
\end{eqnarray*}
 To prove the desired equicontinuity property, we have
to show the existence of a function $\al: \R_+\to\R$ such that $\al(p)\to 0$ as $p\to 0$ and
\begin{equation}\label{eq:18}
\left\|\tilde u(x+h_1,y+h_2,t+h_3)-\tilde u(x,y,t)\right\|\le\al\left(|h|\right)
\end{equation}
uniformly in $\tilde u\in M$ and $h=(h_1,h_2,h_3)\in\R^3$. To achieve (\ref{eq:18}) 
we transform the expression for $\tilde u$ to a convenient form. We make
calculations only for one summand contributing into $\tilde u$ 
(similar argument works for all other summands as well), namely,
\begin{eqnarray}
\label{eq:*}
\lefteqn{
b_{ij}(x,y,t)\int\limits_{0}^xE_r(\xi;x,y,t)
 b_{rq}(\xi,y+\be_r(\xi-x),t+\al_r(\xi-x))\,d\xi}\nonumber\\
\displaystyle
&&\times\int\limits_{\xi}^1E_p(\xi_1;\xi,y+\be_r(\xi-x),t+\al_r(\xi-x))
\nonumber\\
\displaystyle
&&\times  b_{ps}(\xi_1,y+\be_r(\xi-x)+\be_p(\xi_1-\xi),t+\al_r(\xi-x)+\al_p(\xi_1-\xi))\,d\xi_1\\
\displaystyle
&&\times\int\limits_{0}^{\xi_1}E_m(\xi_2;\xi_1,y+\be_r(\xi-x)+\be_p(\xi_1-\xi),t+\al_r(\xi-x)+\al_p(\xi_1-\xi))
\nonumber\\
\displaystyle
&&\times  f_m(\xi_2,y+\be_r(\xi-x)+\be_p(\xi_1-\xi)+\be_m(\xi_2-\xi_1),\nonumber\\
\displaystyle
&&t+\al_r(\xi-x)+\al_p(\xi_1-\xi)+\al_m(\xi_2-\xi_1))\,d\xi_2
\nonumber
\end{eqnarray}
(note that this term is considered up to a multiplicative  constant).
Changing the order of integration, we have
\begin{eqnarray}
\label{eq:19}
\lefteqn{
\int\limits_{0}^xd\xi\int\limits_{\xi}^1d\xi_1\int\limits_{0}^{\xi_1}d\xi_2 =
\int\limits_{0}^xd\xi\left[\int\limits_{0}^{\xi}d\xi_2\int\limits_{\xi}^{1}d\xi_1 +
\int\limits_{\xi}^1d\xi_2\int\limits_{\xi_2}^{1}d\xi_1\right]}\nonumber\\
\displaystyle
&&
=\int\limits_{0}^xd\xi_2\int\limits_{\xi_2}^xd\xi\int\limits_{\xi}^1d\xi_1 +
\int\limits_{0}^xd\xi_2\int\limits_{0}^{\xi_2}d\xi\int\limits_{\xi_2}^1d\xi_1 +
\int\limits_{x}^1d\xi_2\int\limits_{0}^xd\xi\int\limits_{\xi}^1d\xi_1.
\end{eqnarray}
Furthermore, we introduce new variables $\mu$ and $\eta$ (instead of $\xi$ and $\xi_1$) by
\begin{equation}\label{eq:23}
\begin{array}{ccc}
\displaystyle
\mu\equiv y-\be_rx+\xi(\be_r-\be_p)+\xi_1(\be_p-\be_m)\\
\displaystyle
\eta\equiv t-\al_rx+\xi(\al_r-\al_p)+\xi_1(\al_p-\al_m).
\end{array}
\end{equation}
Owing to (\ref{eq:blockB}), the integers $r$, $p$, and $m$ belong to three different sets
$\{1,\dots,l\}$, $\{l+1,\dots,k\}$, and $\{k+1,\dots,n\}$. On the account of
(\ref{eq:ab}), the mapping (\ref{eq:23}) is therefore non-generate.
Apply the mapping  (\ref{eq:23}) to the plane $(\xi,\xi_1)$ and let $\Delta_1$
denote the image of the triangle with vertices  $(\xi_2,\xi_2)$, $(\xi_2,1)$, $(x,1)$.
Similarly, let $\Delta_2$
denote the image of the triangle with vertices $(0,0)$, $(0,1)$, $(x,1)$
and  $\Pi$  be the  image of  the quadrangle   with the 
vertices $(0,\xi_2)$, $(0,1)$, $(\xi_2,1)$, $(\xi_2,\xi_2)$.
By (\ref{eq:19}) and (\ref{eq:23}), the summand (\ref{eq:*}) transforms to
\begin{eqnarray}
\label{eq:**}
\lefteqn{
\int\limits_{0}^xd\xi_2\int\!\!\int_{\Delta_1}F(x,y,t,\xi_2,\mu,\eta)
f_m(\xi_2,\mu,\eta)\,\d\mu\,d\eta}
\nonumber\\
&&
+\int\limits_{0}^xd\xi_2\int\!\!\int_{\Pi}F(x,y,t,\xi_2,\mu,\eta)
f_m(\xi_2,\mu,\eta)\,\d\mu\,d\eta\\
&&
+\int\limits_{x}^1d\xi_2\int\!\!\int_{\Delta_2}F(x,y,t,\xi_2,\mu,\eta)
f_m(\xi_2,\mu,\eta)\,\d\mu\,d\eta,
\nonumber
\end{eqnarray}
where $F$ is a certain function
of $b_{ij}, b_{rq}$, and $b_{ps}$. By assumption (\ref{eq:reg}), $F$ is continuously differentiable in $x,y,t$. 
Our task is therefore reduced to obtaining the estimate (\ref{eq:18}) with $\tilde u$ 
replaced by (\ref{eq:**}). The latter is a straightforward consequence of the fact that the 
lines bounding $\triangle_1$, $\triangle_2$, and $\Pi$ depend linearly on $x,y,t$
(due to the linearity of (\ref{eq:23}) in $x,y,t$). 
The proof is complete.

\begin{rem}\label{rem:reg}
Note that in the course of proving Theorem~\ref{thm:fredh}, we applied 
Theorem~\ref{thm:criter} with $n=3$. 
This choice is essential: a simple analysis of our argument shows that $n=2$ 
would not work. This contrasts to the one-dimensional hyperbolic case where 
$n=2$ makes the job (see \cite{KmRe1,KmRe2}).
In general,
the structure of the regularizer of the problem depends on the number 
of independent variables: 
 for $m$-dimensional hyperbolic PDEs of kind (\ref{eq:10})
we establish the Fredholm property if we regularize the problem by means of the right regularizer 
$\sum\limits_{i=0}^{m}(DC^{-1})^iC$ 
and apply  Theorem~\ref{thm:criter} with $n=m+1$. 
 \end{rem}

\subsection*{Acknowledgments}

I would like to thank Lutz Recke for many useful discussions.


\begin{thebibliography}{10}



\bibitem{ChowHale} S.-N. Chow, J. K. Hale, {\it Methods of Bifurcation Theory},
Grundlehren der Math. Wissenschaften {\bf 251},
Springer-Verlag, New York-Berlin, 1982.

\bibitem{Ki} H. Kielh\"ofer,  {\it Bifurcation Theory. An Introduction with Applications
to PDEs},
Appl. Math. Sciences {\bf 156}, Springer-Verlag, New York-Berlin, 2004.


                                                          
\bibitem{Km} I. Kmit, Smoothing solutions to initial-boundary problems  for first-order hyperbolic systems, 2010
{\it http://arxiv.org/pdf/0908.2189v3}

\bibitem{KmRe1}I.~Kmit and  L.~Recke,
Fredholm Alternative for periodic-Dirichlet problems for
linear hyperbolic systems.
 {\it J.  Math. Anal. and Appl.} {\bf 335} (2007),  355--370.

\bibitem{KmRe2} I. Kmit and L. Recke, Fredholmness and smooth dependence for linear hyperbolic periodic-Dirichlet problems, 2010
{\it http://arxiv.org/pdf/1005.0689v3}

\bibitem{Schechter} M. Schechter  {\it Principles of Functional Analysis}, second ed.,
Graduate Studies in Math. {\bf 36},
American Mathematical Society, Providence, Rhode Island, 2002.


\bibitem{Zeidler}  E. Zeidler, {\it Applied Functional Analysis. Main Principles and their Applications},
Applied Math. Sciences {\bf 109}, Springer-Verlag, Berlin, 1995.



\end{thebibliography}
\end{document}